\documentclass[11pt]{article}
\usepackage{enumerate}
\usepackage{kbordermatrix}
\usepackage{amsfonts}
\usepackage{amsmath}
\usepackage{amssymb}
\usepackage{amstext}
\usepackage{amsxtra}

\newtheorem{theorem}{Theorem}[section]
\newtheorem{lemma}{Lemma}[section]
\newtheorem{corollary}{Corollary}[section]

\newcommand{\C}{\mathbb{C}}
\newcommand{\Z}{\mathbb{Z}}
\newcommand{\Zfour}{\mathbb{Z}_4}
\newcommand{\bfe}{\mathbf{e}}
\newcommand{\Fpm}{\mathbb{F}_{2^m}}
\newcommand{\Fq}{\mathbb{F}_{q}}
\newcommand{\Te}{\mathcal{T}}
\newcommand{\Tem}{\mathcal{T}^*}
\newcommand{\T}{\textup{T}}
\newcommand{\tr}{\textup{tr}}
\newcommand{\R}{\mathcal{R}}
\newcommand{\E}{\mathcal{E}}
\newcommand{\N}{\mathcal{N}}
\newcommand{\sx}{\sum_{X\in\Te}}%
\newcommand{\sy}{\sum_{Y\in\Te}}%
\newcommand{\sz}{\sum_{Z\in\Te}}%
\newcommand{\ssTe}{\sum_{X,Y\in\Te}}%
\newcommand{\sssTe}{\sum_{X,Y,Z\in\Te}}%
\newcommand{\sumu}{\sum_{u\in\Z_4}}%
\newcommand{\suma}{\sum_{a\in R}}%
\newcommand{\sumb}{\sum_{b\in \Fq}}%
\newcommand{\uif}{\textup{ if }}
\newcommand{\DG}{\mathcal{DG}}

\def\whitebox{{\hbox{\hskip 1pt
 \vrule height 6pt depth 1.5pt
 \lower 1.5pt\vbox to 7.5pt{\hrule width
    3.2pt\vfill\hrule width 3.2pt}%
 \vrule height 6pt depth 1.5pt
 \hskip 1pt } }}
\def\qed{\ifhmode\allowbreak\else\nobreak\fi\hfill\quad\nobreak
     \whitebox\medbreak}

\newcommand{\ignore}[1]{}

\begin{document}
\title{Association schemes related to Delsarte-Goethals codes
\thanks{
The research of T. Feng was supported in part by the Fundamental Research Funds for the Central
Universities of China, Zhejiang Provincial Natural Science Foundation (LQ12A01019), National Natural Science
Foundation of China (11201418).
The research of G. Ge was supported by
the National Outstanding Youth Science Foundation of China under Grant
No.~10825103, National Natural Science Foundation of China
under Grant No.~61171198, and Specialized Research Fund for the
Doctoral Program of Higher Education. The research of S. Hu was supported by the Scholarship Award for Excellent
Doctoral Student granted by Ministry of Education of China.
}
}

\author{\small
Tao Feng\thanks{Email: tfeng@zju.edu.cn},
Gennian Ge\thanks{Corresponding author. Email: gnge@zju.edu.cn}
and Sihuang Hu\thanks{Email: husihuang@zju.edu.cn}\\
\small          Department of Mathematics\\
\small          Zhejiang  University\\
\small          Hangzhou 310027, Zhejiang\\
\small          P. R. China\\
}
\date{}\maketitle

\begin{abstract}
In this paper, we construct an infinite series of $9$-class association schemes from a
refinement of the partition of Delsarte-Goethals codes by their Lee
weights. The explicit expressions of the dual schemes are  determined
through direct manipulations of complicated exponential sums.
As a byproduct, the other three infinite families of association schemes
are also obtained as fusion schemes and quotient schemes.

\smallskip
\end{abstract}

\section{Introduction}\label{intro}
Since the discovery of the $\Z_4$-linearity of Kerdock, Preparata,
Goethals, and Delsarte-Goethals codes (see~\cite{HKCSS}), there have been several applications
of $\Z_4$-linear codes in the constructions of combinatorial configurations,
such as $t$-designs and association schemes.
According to the paper \cite{So} of Sol\'e,
the discovery of the $\Z_4$-linearity is motivated by a construction of
association schemes due to Liebler and Mena (see~\cite{LM}) using Galois rings of characteristic $4$.
The research on constructing $t$-designs from linear codes over $\Z_4$
is first put forwarded by Harada in~\cite{Ha}, and subsequently followed by
Helleseth and his collaborators in a series of papers (see~\cite{SKH} and
the references  therein).

Association schemes form a central part of algebraic combinatorics, and have
played important roles in several branches of mathematics, such as coding
theory and graph theory.
Henry Cohn and his collaborators in
\cite{BBC} conjectured  that a certain $3$-class association scheme on
$64$ points determines a universally optimal configuration in
$\mathbb{R}^{14}$. Then in \cite{ABS} Abdukhalikov, Bannai and Suda
generalized this example in terms of binary and quaternary Kerdock and
Preparata codes, as well as in terms of maximal set of mutually unbiased
basis. To be specific, they constructed a series of $3$-class association schemes
using the partition of shortened Kerdock codes induced by their Lee
weights, and dual schemes on the cosets of punctured Preparata codes.
This motivates us to look at another important class of quaternary codes,
the Delsarte-Goethals ($\DG$) codes.

The $\DG$ code has  $6$ nonzero Lee weights. It turns out that the
partition of the $\DG$ code by its Lee weights does not give an
association scheme. We should further pick out the codewords of the
Kerdock code, and this eventually yields a $9$-class association scheme.
Using complicated exponential sums and heavy computations, we get the
structure of the dual scheme and their eigenmatrices. The dual scheme
can not be obtained in an analogous way on the cosets of the Goethals
code. There is a conjectured $22$-class association scheme on the $\DG$
code by the partition using its complete weight enumerator, but our
scheme is not a fusion scheme of this conjectured scheme.

When $m=3$, the image of the $\DG$ code under Gray map is linear, and we
have checked that the image of our $9$-class association scheme under Gray map
remains to be a scheme but with different parameters.
When $m>3$, the Gray map image of the $\DG$ code is no longer nonlinear.
It is not clear whether we can find translate schemes in
elementary abelian $2$-groups with the same parameters as the schemes we construct in this paper.

This paper is organized as follows. In Section~\ref{sec:prel}, we give
some preliminaries on association schemes, Galois rings, and some quaternary
codes with their Lee weight distributions.
In Section~\ref{sec:descriptionOfDGSchemes}, we describe our
construction of a $9$-class association scheme coming from refining
the partition of the $\DG$ code by its Lee weights.
Its dual scheme and eigenmatrices are explicitly determined.
Also a $7$-class scheme as its fusion scheme and a $5$-class scheme
as its quotient scheme are obtained.
The eigenmatrices of these schemes are listed in Appendix A.
The proof of our main result is provided in Section~\ref{sec:completion}.

\section{Preliminaries}\label{sec:prel}
\subsection{Association schemes.}
Let $X$ be a nonempty finite set, and a set of symmetric relations $R_0,R_1,\cdots, R_d$ be a partition of $X\times X$ such that $R_0=\{(x,x)|x\in X\}$. Denote by $A_i$ the adjacency matrix of $R_i$ for each $i$, whose $(x,y)$-th entry is $1$ if $(x,y)\in R_i$ and $0$ otherwise. We call $(X,\{R_i\}_{i=0}^d)$ a {\it $d$-class association scheme} if there exist numbers $p_{i,j}^k$ such that
\[
A_iA_j=\sum_{k=0}^dp_{i,j}^kA_k.
\]
These numbers are called the intersection numbers of the scheme. The $\C$-linear span of $A_0,A_1,\cdots,A_d$ forms a semisimple  algebra of dimension $d+1$, called the {\it Bose-Mesner algebra} of the scheme. With respect to the basis $A_0,A_1,\cdots,A_d$, the matrix of the multiplication by $A_i$ is denoted by $B_i$, namely
\[
A_i(A_0,A_1,\cdots,A_d)=(A_0,A_1,\cdots,A_d)B_i,\;0\leq i\leq d.
\]
Since each $A_i$ is symmetric, this algebra is commutative. There exists a set of minimal idempotents $E_0,E_1,\cdots,E_d$  which also forms a basis of the algebra. The $(d+1)\times (d+1)$ matrix $P$ such that
\[
(A_0,A_1,\cdots,A_d)=(E_0,E_1,\cdots,E_d)P
\]
is called the {\it first eigenmatrix} of the scheme. Dually, the $(d+1)\times (d+1)$ matrix $Q$ such that
\[
(E_0,E_1,\cdots,E_d)=\frac{1}{|X|}(A_0,A_1,\cdots,A_d)Q
\]
is called the {\it second eigenmatrix} of the scheme. Clearly, we  have $PQ=|X|I$.

We call an association scheme $(X,\{R_i\}_{i=0}^d)$ a {\it translation association scheme} or a {\it Schur ring} if $X$ is a  (additively
written) finite abelian  group and there exists a partition $S_0=\{0\},S_1,\cdots,S_d$  of $X$ such that
\[
R_i=\{(x,x+y)|\,x\in X, y \in S_i\}.
\]
For brevity, we will just say that $(X,\{S_i\}_{i=0}^d)$ is an association scheme.

Assume that $(X,\{S_i\}_{i=0}^d)$ is a translation association scheme. There is an equivalence relation defined on the character group $\hat{X}$ of $X$ as follows: $\chi\sim\chi'$ if and only if $\chi(S_i)=\chi'(S_i)$ for each $0\leq i\leq d$. Here $\chi(S)=\sum_{g\in S}\chi(g)$, for any $\chi\in\hat{X}$, and $S\subseteq X$. Denote by $D_0, D_1,\cdots,D_d$ the equivalence classes, with $D_0$ consisting of only the principal character. Then $(\hat{X},\{D_i\}_{i=0}^d)$ forms a translation association scheme, called the {\it dual} of $(X,\{S_i\}_{i=0}^d)$. The first eigenmatrix of the dual scheme is equal to the second eigenmatrix of the original scheme. Please refer to \cite{BI} and \cite{BCN} for more details.


We shall need the following well-known criterion due to Bannai \cite{Ban} and Muzychuk \cite{Mu}, called the {\it Bannai-Muzychuk criterion}: {\it Let $P$ be the first eigenmatrix of an association scheme $(X, \{R_i\}_{0\leq i\leq d})$, and $\Lambda_0:=\{0\}, \Lambda_1,\ldots ,\Lambda_{d'}$ be a partition of $\{0,1,\ldots ,d\}$. Then $(X, \{R_{\Lambda_i}\}_{0\leq i\leq d'})$ forms an association scheme if and only if there exists a partition $\{\Delta_i\}_{0\leq i\leq d'}$ of $\{0,1,2,\ldots ,d\}$ with $\Delta_0=\{0\}$ such that each $(\Delta_i, \Lambda_j)$-block of $P$ has a constant row sum. Moreover, the constant row sum of the $(\Delta_i, \Lambda_j)$-block is the $(i,j)$-th entry of the first eigenmatrix of the fusion scheme.}

\subsection{Quaternary codes}
Let $\mu:\Zfour\rightarrow \Z_2$
denote the modulo $2$ reduction map. We extend $\mu$ to $\Z_4[x]$ in the
natural way. A monic polynomial $g(x)\in\Z_4[x]$  is said to be
basic irreducible if $\mu(g(x))$ is a monic irreducible polynomial  in
$\Z_2[x]$.  A Galois ring $R=GR(4,m)$ of characteristic $4$ with $4^m$
elements is defined as the quotient ring $\Z_4[x]/(f(x))$,  where $f(x)$ is a
monic basic irreducible polynomial of degree $m$. The collection of
non-units of $R$ forms the unique maximal ideal $2R$, so $R$ is a local
ring. Clearly, $\mu$ has a natural extension to $R[x]$ and
$\mu(R)=R/2R$ is isomorphic to a finite field $\Fq$ of size $q=2^m$.

As a multiplicative group, the units $R^*$ of $R$ has a cyclic subgroup of order $2^m-1$,
whose generator is denoted by $\beta$. Let $\Te=\{0,1,\beta,\ldots,\beta^{2^m-2}\}$.
Every element $z\in R$ can be expressed uniquely as
$$z=A+2B, \qquad A,B\in\Te.$$
Let $\mu(\beta)=\alpha.$ Then $\alpha$ is a primitive element
in $\Fq$, and $\mu(\Te)=\Fq$.

The Galois ring $R$ has a cyclic Galois group
of order $m$ generated by the following Frobenius map $\sigma$:
$$\sigma(z)=\sigma(A+2B)=A^2+2B^2.$$
The trace of $z$, $\T(z)$, from $R$ to $\Z_4$ is defined as
$$\T(z)=\sum_{i=0}^{m-1}\sigma^i(z),$$
and $\tr(x)$ is the ordinary trace function from $\Fq$ to $\Z_2$.

The Goethals code $\mathcal{G}$ of length $q=2^m$ over $\Z_4$ is a linear code
with the following parity-check matrix
\begin{eqnarray*}
  H_\mathcal{G}=\left[
    \begin{array}{llllll}
    1 &1 &1 &1 & \cdots &1\\
    0 &1 &\beta&\beta^2&\cdots&\beta^{2^m-2}\\
    0 &2 &2\beta^3&2\beta^6&\cdots&2\beta^{3(2^m-2)}\\
  \end{array}
  \right].
\end{eqnarray*}
It is shown in \cite[Theorem 25]{HKCSS} that if $m$ is odd, then the
Goethals code $\mathcal{G}$ has minimum Lee distance $8$. The Delsarte-Goethals
code $\mathcal{DG}$ is defined as the dual of $\mathcal{G}$ over $\Z_4$,
so its generator matrix is just $H_\mathcal{G}$.

Also the Delsarte-Goethals code has the following trace description.
Let $\mathbf{c}(u,a,b)$, where $u\in\Z_4,a\in R,b\in\Te$, be a vector in $\Z_4^q$ indexed
by the elements of $\Te$ such that $\mathbf{c}(u,a,b)_x=u+\T(ax+2bx^3)$ for all $x\in\Te$,
then
$$\mathcal{DG}=\{\mathbf{c}(u,a,b)\mid u\in\Z_4,a\in R,b\in\Te\}.$$
The Lee weight of codeword $\mathbf{c}(u,a,b)$ can be expressed as
\begin{equation}
w_L(\mathbf{c}(u,a,b))=q-\Re\left(i^u\sum_{x\in\Te} i^{\T(ax+2bx^3)}\right).
  \label{weightDisCharValue}
\end{equation}

\begin{lemma}\cite[Theorem 1]{HKS}
  Let $q=2^m$ where $m$ is odd. The Lee weight distribution of the Delsarte-Goethals code
  $\mathcal{DG}$ is
  $$
  A_j=\left\{
    \begin{array}{ll}
      1, &\uif j=0 \textup{ or } 2q;\\
      (q-1)q(2q-1)/6, &\uif j=q\pm\sqrt{2q};\\
      (q-1)2q(q+4)/2, &\uif j=q\pm\sqrt{q/2};\\
      (2q-1)(q^2-q+2), &\uif j=q.
    \end{array}
    \right.
  $$
  \label{weightGoethals}
\end{lemma}

The quaternary Kerdock code $\mathcal{K}$ of length $q=2^m$ is
a subcode of $\mathcal{DG}$ generated by
\begin{eqnarray*}
  H_\mathcal{K}=\left[
    \begin{array}{llllll}
   1 &1 &1 &1 & \cdots &1\\
    0 &1 &\beta&\beta^2&\cdots&\beta^{2^m-2}\\
 \end{array}
  \right].
\end{eqnarray*}
Hence $$\mathcal{K}=\{\mathbf{c}(u,a,0)\mid u\in\Z_4,a\in R\}.$$
The quaternary Preparata code $\mathcal{P}$ is a code with parity-check matrix
$H_\mathcal{K}$. If $m$ is odd, it has minimum Lee weight 6.

\begin{lemma}\cite[Lemma 1]{HKS}
  Let $q=2^m$ where $m\ge 3$ is odd. The Lee weight distribution of the
  Kerdock code $\mathcal{K}$ is
  $$
  A_j=\left\{
    \begin{array}{ll}
      1, &\uif j=0 \textup{ or } 2q;\\
      2q(q-1), &\uif j=q\pm \sqrt{q/2};\\
      4q-2,&\uif j=q.
    \end{array}\right.
  \label{weightKerdock}
  $$
\end{lemma}

The following relations on the Kerdock code $\mathcal{K}$ will determine an
abelian $4$ class association scheme:
\begin{equation}
(x,y)\in\left\{
  \begin{array}{ll}
    R_0, &\uif w_L(x-y)= 0;\\
    R_1, &\uif w_L(x-y)= q-\sqrt{q/2};\\
    R_2, &\uif w_L(x-y)= q;\\
    R_3, &\uif w_L(x-y)= q+\sqrt{q/2};\\
    R_4, &\uif w_L(x-y)= 2q.\\
  \end{array}
  \right.
  \label{relationKerdock}
\end{equation}

\begin{theorem}
  The relations (\ref{relationKerdock}) on the codewords of quaternary Kerdock
  code define a $4$ class abelian association scheme.
  \label{assoScheKer}
\end{theorem}

The above scheme is just the dual of the scheme constructed in
\cite[Proposition 6]{BD}.

\section{Schemes related to the $\mathcal{DG}$ code}\label{sec:descriptionOfDGSchemes}
It is natural to ask whether the obvious generalizations of the relations~(\ref{relationKerdock})
on the Delsarte-Goethals code $\mathcal{DG}$ will also give an abelian association
scheme or not. It turns out that the answer is no.
In order to get an association scheme, we should modify the relations a little:
\begin{equation}
(x,y)\in\left\{
  \begin{array}{ll}
    S_0, &\uif w_L(x-y)= 0;\\
    S_1, &\uif w_L(x-y)= q-\sqrt{2q};\\
    S_2, &\uif w_L(x-y)= q-\sqrt{q/2};\\
    S_3, &\uif w_L(x-y)= q \textup{ and } x-y\in\mathcal{K};\\
    S_4, &\uif w_L(x-y)= q \textup{ and } x-y\notin\mathcal{K};\\
    S_5, &\uif w_L(x-y)= q+\sqrt{q/2};\\
    S_6, &\uif w_L(x-y)= q+\sqrt{2q};\\
    S_7, &\uif w_L(x-y)= 2q.\\
  \end{array}
  \right.
  \label{relation7class}
\end{equation}


\begin{theorem}
  The relations (\ref{relation7class}) on the codewords of
  quaternary Delsarte-Goethals code define a $7$ class
  abelian association scheme $\mathfrak{A}$. The first and the
  second eigenmatrices are given in Appendix A.
  \label{thm:assoSche7class}
\end{theorem}

This is not the end of the story. Actually we can provide a more refined
description of relations (\ref{relation7class}) to get an abelian $9$ class
association scheme as follows:
\begin{equation}
(x,y)\in\left\{
  \begin{array}{ll}
    S_0, &\uif w_L(x-y)= 0;\\
    S_1, &\uif w_L(x-y)= q-\sqrt{2q};\\
    S'_{21}, &\uif w_L(x-y)= q-\sqrt{q/2} \textup{ and } x-y\in\mathcal{K};\\
    S'_{22}, &\uif w_L(x-y)= q-\sqrt{q/2} \textup{ and } x-y\notin\mathcal{K};\\
    S_3, &\uif w_L(x-y)= q \textup{ and } x-y\in\mathcal{K};\\
    S_4, &\uif w_L(x-y)= q \textup{ and } x-y\notin\mathcal{K};\\
    S'_{51}, &\uif w_L(x-y)= q+\sqrt{q/2} \textup{ and } x-y\notin\mathcal{K};\\
    S'_{52}, &\uif w_L(x-y)= q+\sqrt{q/2} \textup{ and } x-y\in\mathcal{K};\\
    S_6, &\uif w_L(x-y)= q+\sqrt{2q};\\
    S_7, &\uif w_L(x-y)= 2q.\\
  \end{array}
  \right.
  \label{relation9class}
\end{equation}

\begin{theorem}
  The relations (\ref{relation9class}) on the codewords of quaternary Delsarte-Goethals
  code define a $9$ class abelian association scheme $\mathfrak{B}$. The first and the
  second eigenmatrices are given in Appendix A.
  \label{thm:assoSche9class}
\end{theorem}

We leave the proofs of the above two theorems to the end
of this section, since they are immediate outcomes of
Theorem~\ref{main}.
Recall the trace description of the Delsarte-Goethals code $\mathcal{DG}$,
we know there is a one-to-one correspondence between the codewords of $\mathcal{DG}$
and the set $\Z_4\times R\times\Te$ given by
$(u,a,b)\longleftrightarrow \mathbf{c}(u,a,b).$
Since $\mu(\Te)=\Fq$, there is a group isomorphism between the group
$G=\Z_4\times R\times \Fq$ and the $\mathcal{DG}$ code.
For $(u,a,b)\in G$, we introduce an exponential sum
$$S(u,a,b)=\sum_{X\in\Te}i^{u+\T(aX+2bX^3)}+\sum_{X\in\Te}i^{-u-\T(aX+2bX^3)},$$
here we have identified the element $b\in\Fq$ with its pre-image $\mu^{-1}(b)\in\Te$.
Now Eqn. (\ref{weightDisCharValue}) becomes
\begin{equation}
w_L(\mathbf{c}(u,a,b))=q-S(u,a,b)/2.
  \label{weightDisCharValue2}
\end{equation}
So we can see that for $(u,a,b)\in G,$
$$
S(u,a,b)\in\{\pm2q,\pm2\sqrt{2q},\pm\sqrt{2q},0\},
$$
from the weight distribution of Delsarte-Goethals code.
According to the value of $S(u,a,b)$, we give a partition of $G$ into ten parts
as follows:
$$
\begin{array}{l}
  \R_0=\{(u,a,b)\in G\mid S(u,a,b)=2q\}=\{(0,0,0)\},\\
  \R_1=\{(u,a,b)\in G\mid S(u,a,b)=2\sqrt{2q}\},\\
  \R_2=\{(u,a,b)\in G\mid S(u,a,b)=\sqrt{2q}, b=0\},\\
  \R_3=\{(u,a,b)\in G\mid S(u,a,b)=\sqrt{2q}, b\neq0\},\\
  \R_4=\{(u,a,b)\in G\mid S(u,a,b)=0, b=0\},\\
  \R_5=\{(u,a,b)\in G\mid S(u,a,b)=0, b\neq0\},\\
  \R_6=\{(u,a,b)\in G\mid S(u,a,b)=-\sqrt{2q}, b\neq0\},\\
  \R_7=\{(u,a,b)\in G\mid S(u,a,b)=-\sqrt{2q}, b=0\},\\
  \R_8=\{(u,a,b)\in G\mid S(u,a,b)=-2\sqrt{2q}\},\textup{  and }\\
  \R_9=\{(u,a,b)\in G\mid S(u,a,b)=-2q\}=\{(2,0,0)\}.\\\
\end{array}
$$
Since the group $G$ is abelian, its character group $\widehat{G}\cong G.$
In order to describe the dual association scheme, we shall first provide the
dual partition of the group $\widehat{G}=G$.
For convenience, we will use capital letters such as $X,Y$ to denote elements in $\Te$,
and the corresponding lower cases  to represent their respective projections modulo $2$ in $\Fq$.
The dual partition is as follows:
\begin{eqnarray*}
  \E_0&=& \left\{ (0,0,0) \right\},\\
  \E_1&=& \left\{ (0,0,r)\mid r\in \Fq^* \right\},\\
  \E_2&=& \left\{ (1,X,x^3)\mid X\in\Te\right\}\cup\left\{ (3,-X,x^3)\mid X\in\Te \right\},\\
  \E_3&=& (\{(1,X,r)\mid X\in\Te,r\in\Fq\}\cup\{ (3,-X,r)\mid X\in\Te,r\in\Fq\})\backslash \E_2,\\
  \E_4&=& \quad\{(0,-X+Y,x^3+y^3)\mid X,Y\in\Te,X\neq Y\}\\
     && \cup\ \{(2,X+Y,x^3+y^3)\mid X,Y\in\Te\}\\
     && \cup\ \{(2,-X-Y,x^3+y^3)\mid X,Y\in\Te\},\\
  \E_5&=&(\{(0,S,r)\mid S\in R\backslash 2R,r\in\Fq\}\cup\{(2,S,r)\mid S\in R,r\in\Fq\})\backslash \E_4,\\
  \E_6&=&\quad \{(1,2X-Y,y^3)\mid X,Y\in\Te,X\neq Y\}\\
    &&\cup\ \{(1,-X-Y-Z,x^3+y^3+z^3)\mid X,Y,Z\in\Te\textup{ are pairwise distinct}\}\\
    &&\cup\ \{(1,X+Y-Z,x^3+y^3+z^3)\mid X,Y,Z\in\Te\textup{ are pairwise distinct}\}\\
    &&\cup\ \{(3,2X+Y,y^3)\mid X,Y\in\Te,X\neq Y\}\\
    &&\cup\ \{(3,X+Y+Z,x^3+y^3+z^3)\mid X,Y,Z\in\Te\textup{ are pairwise distinct}\}\\
    &&\cup\ \{(3,-X-Y+Z,x^3+y^3+z^3)\mid X,Y,Z\in\Te\textup{ are pairwise distinct}\},
\end{eqnarray*}
\begin{eqnarray*}
  \E_7&=&\quad \{(1,2X-Y,r)\mid X,Y\in\Te,X\neq Y,r\neq y^3\}\\
    &&\cup\ \{(1,-X-Y-Z,r)\mid X,Y,Z\in\Te\textup{ are pairwise distinct},r\neq x^3+y^3+z^3\}\\
    &&\cup\ \{(1,X+Y-Z,r)\mid X,Y,Z\in\Te\textup{ are pairwise distinct},r\neq x^3+y^3+z^3\}\\
    &&\cup\ \{(3,2X+Y,r)\mid X,Y\in\Te,X\neq Y,r\neq y^3\}\\
    &&\cup\ \{(3,X+Y+Z,r)\mid X,Y,Z\in\Te\textup{ are pairwise distinct},r\neq x^3+y^3+z^3\}\\
    &&\cup\ \{(3,-X-Y+Z,r)\mid X,Y,Z\in\Te\textup{ are pairwise distinct},r\neq x^3+y^3+z^3\},\\
  \E_8&=& \{(0,2X,\sum_{i=1}^4y_i^3)|X\in\Tem,Y_i\in\Te,2X=\sum_{i=1}^4Y_i \textup{ or } 2X=Y_1+Y_2-Y_3-Y_4\},\textup{ and}\\
  \E_9&=&\{(0,2X,r)\mid X\in\Tem, r\in \Fq\}\backslash \E_8.
\end{eqnarray*}

\noindent{\bf Remark:} Let $m$ be an odd integer. In the study of
dimensional dual hyperovals, Pasini and Yoshiara
\cite[Proposition 1.7]{PY} showed that the
Cayley graph of the following set
\[
S=\{(1,x,x^3)|\,x\in\Fq\}\subset \Z_2\times \Fq\times \Fq
\]
is a distance regular graph of diameter $4$. This is analogous to our set
$\E_2$ above.

\begin{theorem}\label{main}
  Let $G=\Z_4\times R\times\Fq$, and define the binary relations
  $R_i=\{(g,h)\mid g-h\in \R_i\}$ for $i=0,\ldots,9$. Then
  $\mathfrak{B}'=(G;R_i,0\leq i\leq9)$ is a $9$-class association scheme,
  with the first and the second eigenmatrices given by $P$ and $Q$
  (listed in Appendix A). The binary relations $R'_i=\{(g,h)\mid g-h\in \E_i\}$
  for $i=0,\ldots,9$ define an association scheme which
  is dual to $\mathfrak{B}'$, so it has the first
  and second eigenmatrices: $P'=Q,Q'=P.$
\end{theorem}
\begin{proof}
Denote $s=\sqrt{2q}.$ The element of the group ring $\mathbb{C}G$ will
be written as $\sum_{(u,a,b)\in G}c(u,a,b)[(u,a,b)]$, where
$c(u,a,b)\in \mathbb{C}$. We set
\begin{eqnarray*}
  \N_{2i}&=& \sum_{(u,a,b)\in G}S(u,a,b)^i[(u,a,b)], \textup{ and}\\
  \N_{2i+1}&=& \sum_{(u,a,b)\in G\atop b=0}S(u,a,b)^i[(u,a,b)],
\end{eqnarray*}
for $i=0,1,\ldots,4.$ The above transformation can be written in the matrix form as
$$
(\N_0,\N_1,\ldots,\N_9)=(\R_0,\R_1,\ldots,\R_9)\,T,
$$
where
\begin{small}
$$
T=\left( \begin {array}{cccccccccc} 1&1&{s}^{2}&{s}^{2}&{s}^{4}&{s}^{4}
&{s}^{6}&{s}^{6}&{s}^{8}&{s}^{8}\\ \noalign{\medskip}1&0&2\,s&0&4\,{s}
^{2}&0&8\,{s}^{3}&0&16\,{s}^{4}&0\\ \noalign{\medskip}1&1&s&s&{s}^{2}&
{s}^{2}&{s}^{3}&{s}^{3}&{s}^{4}&{s}^{4}\\ \noalign{\medskip}1&0&s&0&{s
}^{2}&0&{s}^{3}&0&{s}^{4}&0\\ \noalign{\medskip}1&1&0&0&0&0&0&0&0&0
\\ \noalign{\medskip}1&0&0&0&0&0&0&0&0&0\\ \noalign{\medskip}1&0&-s&0&
{s}^{2}&0&-{s}^{3}&0&{s}^{4}&0\\ \noalign{\medskip}1&1&-s&-s&{s}^{2}&{
s}^{2}&-{s}^{3}&-{s}^{3}&{s}^{4}&{s}^{4}\\ \noalign{\medskip}1&0&-2\,s
&0&4\,{s}^{2}&0&-8\,{s}^{3}&0&16\,{s}^{4}&0\\ \noalign{\medskip}1&1&-{
s}^{2}&-{s}^{2}&{s}^{4}&{s}^{4}&-{s}^{6}&-{s}^{6}&{s}^{8}&{s}^{8}
\end {array} \right).
$$
\end{small}
Using Maple, it is easy to compute that
$\det(T)=-1152\,{s}^{23} \left( s-1 \right) ^{2} \left( s+1 \right) ^{2},$
so $T$ is invertible.
The key step in our proof is the completion of the
following character table $\mathfrak{T}$:
\renewcommand{\kbldelim}{(}
\renewcommand{\kbrdelim}{)}
\textrm{\begin{tiny}
$$
\kbordermatrix{\mbox{}
&\N_0&\N_1&\N_2&\N_3&\N_4&\N_5&\N_6&\N_7&\N_8&\N_9\\
\E_0&4\,{q}^{3}&4\,{q}^{2}&0&0&8\,{q}^{4
}&8\,{q}^{3}&0&0&16\,{q}^{4}\left( 3\,q-1 \right)
&16\,{q}^{3} \left( 3\,q-1 \right) \\
\E_1&   0&4\,{q}^{2}&0&0&0&8\,{q}^
{3}&0&0&0&16\,{q}^{3} \left( 3\,q-1 \right) \\
\E_2&  0&0&4
\,{q}^{3}&4\,{q}^{2}&0&0&8\,{q}^{3} \left( 3\,q-1 \right) &8\,{q}^{2}
 \left( 3\,q-1 \right) &0&0\\
\E_3&   0&0&0&4\,{q}^{2}&0&0&0
&8\,{q}^{2} \left( 3\,q-1 \right) &0&0\\
\E_4&  0&0&0&0&8\,
{q}^{3}&8\,{q}^{2}&0&0&32\,{q}^{3} \left( 3\,q-2 \right)&32\,{q}^{4}\\
\E_5&  0&0&0&0&0&8\,{q}^{2}&0&0&32\,{q}^{3}
 \left( q-2 \right)&32\,{q}^{4} \\
\E_6&   0&0&0&0&0&0&24\,{q}^{3}&8\,{q
}^{2} \left( 2\,q-1 \right) &0&0\\
\E_7&  0&0&0&0&0&0&0&8\,
{q}^{2} \left( 2\,q-1 \right) &0&0\\
\E_8&  0&0&0&0&0&0&0&0
&48\,{q}^{4}&16\,{q}^{3} \left( 2\,q-1 \right) \\
\E_9&  0&0
&0&0&0&0&0&0&0&16\,{q}^{3} \left( 2\,q-1 \right) \\
}.
$$
\end{tiny}}

The next section is devoted to the computation of this table. Taking the
column indexed by $\N_6$ as an illustration,  we can see that
\begin{eqnarray*}
g( \N_6)&=&\left\{
\begin{array}{ll}
    8\,{q}^{3} \left( 3\,q-1 \right), &\uif g\in\E_2;\\
    24\,{q}^{3}, &\uif g\in\E_6;\\
    0, &\textup{ otherwise}.
  \end{array}
  \right.
\end{eqnarray*}
Thus the character table $P$ is obtained by multiplying the above
character table $\mathfrak{T}$ with the matrix $T^{-1}$ from the left.
Now the assertion that $\mathfrak{B}'=(G;R_i,0\leq i\leq9)$ is a
$9$-class association scheme is an immediate consequence of the
Bannai-Muzychuk criterion \cite{Ban,Mu}.
\end{proof}\qed

\medskip

{\em Proof of Theorems~\ref{thm:assoSche7class}-\ref{thm:assoSche9class}.}
From the statements at the beginning of this section,
Theorem~\ref{thm:assoSche9class} is clear. Theorem~\ref{thm:assoSche7class}
follows directly from the Bannai-Muzychuk criterion and the eigenmatrices
of association scheme $\mathfrak{B}$. The scheme $\mathfrak{A}$ is a fusion
scheme of $\mathfrak{B}$.
\qed

\begin{corollary}\label{coro:assoSche5and4class}
  There exists a 5-class association scheme $\mathcal{C}$ on the quotient
  group $G/\langle (2,0,0)\rangle$. Furthermore, there exists a
  4-class fusion scheme $\mathcal{D}$ of the scheme $\mathcal{C}$.
  Their first and second eigenmatrices are given in Appendix A.
\end{corollary}
\begin{proof}
  This can be readily checked using the eigenmatrices of the association scheme
  $\mathcal{B'}$ with the help of Bannai-Muzychuk criterion.
\end{proof}\qed

\section{Completion of the character table $\mathfrak{T}$}
\label{sec:completion}
\subsection{Columns indexed by $\N_0,\N_1,\N_2$ and $\N_3$}
The first four columns of the character table $\mathfrak{T}$ can be
obtained through direct computations. Here we only provide the proof of
the column indexed by $\N_2$, each of the other three cases is exactly the same.

For $g=(v,c,d)\in G$, we have
\begin{eqnarray*}
g(\N_2)
      &=&\sum_{(u,a,b)\in G}S(u,a,b)i^{uv+\T(ac+2bd)}\\
      &=&\sum_{(u,a,b)\in G}\sum_{X\in\Te}i^{u(v+1)+\T(a(c+X))+2\T(b(d+X^3))}+\\
      &&\sum_{(u,a,b)\in G}\sum_{X\in\Te}i^{u(v-1)+\T(a(c-X))+2\T(b(d-X^3))}\\
      &=& \sum_{X\in\Te}\sum_{u\in\Z_4}i^{u(v+1)}\sum_{a\in R}i^{\T(a(c+X))}\sum_{b\in\Te}i^{2\T(b(d+X^3))}+\\
      &&  \sum_{X\in\Te}\sum_{u\in\Z_4}i^{u(v-1)}\sum_{a\in R}i^{\T(a(c-X))}\sum_{b\in\Te}i^{2\T(b(d-X^3))}\\\\
      &=&\left\{ \begin{array}{ll}
                  4\,q^3, &\textup{ if } g\in \E_2;\\
                  0,   &\textup{ otherwise},
                \end{array}
        \right.
\end{eqnarray*}
and this completes our proof.

\subsection{Columns indexed by $\N_4,\N_5,\N_6,\N_7$ and $\N_9$}
First, a direct calculation gives that
\begin{eqnarray*}
g(\N_4)&=&\sum_{(u,a,b)\in G}S(u,a,b)^2i^{uv+\T(ac+2bd)}\\
&=&\sumu i^{u(v+2)}\ssTe\suma i^{\T(a(c+X+Y))}\sumb i^{\T(2b(d+X^3+Y^3))}+\\
&&\sumu i^{u(v-2)}\ssTe\suma i^{\T(a(c-X-Y))}\sumb i^{\T(2b(d-X^3-Y^3))}+\\
&&2\sumu i^{uv}\ssTe\suma i^{\T(a(c+X-Y))}\sumb i^{\T(2b(d+X^3-Y^3))}
\end{eqnarray*}
for $g=(v,c,d)\in G.$

Suppose that $v=0$ and $c=-Z+W$ for some $Z,W\in\Te$ with $Z\neq W$. Then
only the last term in the above sum possibly contributes. Using $(b)$ of
Lemma~\ref{squareCount},
we see that it will contribute $8q^3$ if $d=z^3+w^3$ and zero otherwise.
A routine analysis shows that
\begin{eqnarray*}
g(\N_4)&=&\left\{
  \begin{array}{ll}
    8\,q^4, &\uif g\in\E_0;\\
    8\,q^3, &\uif g\in\E_4;\\
    0,     &\textup{ otherwise},
  \end{array}
  \right.
\end{eqnarray*}
and this completes the column indexed by $\N_4.$

The column indexed by $\N_5$ can be checked exactly as same as $\N_4$.

Now we continue to treat the column $\N_6$:
\begin{eqnarray*}
g(\N_6)&=&\sum_{(u,a,b)\in G}S(u,a,b)^3i^{uv+\T(ac+2bd)}\\
&=&\sumu i^{u(v+3))}\sssTe\suma i^{\T(a(c+X+Y+Z))}\sumb i^{\T(2b(d+X^3+Y^3+Z^3))}+\\
&&\sumu i^{u(v-3))}\sssTe\suma i^{\T(a(c-X-Y-Z))}\sumb i^{\T(2b(d-X^3-Y^3-Z^3))}+\\
&&3\sumu i^{u(v+1)}\sssTe\suma i^{\T(a(c+X+Y-Z))}\sumb i^{\T(2b(d+X^3+Y^3-Z^3))}+\\
&&3\sumu i^{u(v-1)}\sssTe\suma i^{\T(a(c-X-Y+Z))}\sumb i^{\T(2b(d-X^3-Y^3+Z^3))}.\\\\
\end{eqnarray*}
First, suppose that $g=(1,W,w^3)\in\E_2$ for some $W\in\Te.$
Then only the first and last terms in the above sum will contribute.
Using $(d)$ of Lemma~\ref{squareCount},
we see that the first term will contribute $4\,q^3$ if $d=w^3$
and zero otherwise. Similar analysis shows that the last term will
contribute $12\,q^3(2q-1)$ if $d=w^3$ and zero otherwise by $(c)$
of Lemma~\ref{squareCount}.
The analysis for $g=(3,-W,w^3)\in\E_2$ is exactly the same.
So $g(\N_6)=8q^3(3q-1)$ for $g\in\E_2$.
Secondly, when $g\in\E_6,$ the argument can be proved similarly
using Lemmas~\ref{lemma:N_6_1}-\ref{lemma:N_6_2} in Appendix B.

The column indexed by $\N_7$ and $\N_9$ are readily completed by using
Lemma~\ref{cubicCount} and Corollary~\ref{coro: quaCount} respectively.

\subsection{Column indexed by $\N_8$}
This is the most difficult case.
When $g=(0,0,0)$,
the identity $g(\N_8)=16q^4(3q-1)$ is direct. When $g\in\E_8,$
we can prove $g(\N_8)=48q^4$ by
Corollaries~\ref{coro: N_8_1}-\ref{coro: N_8_2}.
The exponential sum $\xi(a,b)=\sx i^{\T(aX+2bX^3)}$ is closely
related to $S(u,a,b)$.
We further introduce two exponential sums
$$\mathbf{E}(c,d):=\suma\sumb\left(\xi^4(a,b)+\overline{\xi^4(a,b)}
    +6\,\xi^2(a,b)\overline{\xi^2(a,b)}\right)i^{\T(ac+2bd)}$$
and
$$\mathbf{F}(c,d):=\suma\sumb\left(\xi^3(a,b)\overline{\xi^(a,b)}
+\xi(a,b)\overline{\xi^3(a,b)}\right)i^{uv+\T(ac+2bd)}.$$
Then
\begin{eqnarray*}
g(\N_8)=\sum_{(u,a,b)\in G}S(u,a,b)^4i^{uv+\T(ac+2bd)}
=\sumu i^{uv}\mathbf{E}(c,d)+\,4\sumu i^{u(v+2)}\mathbf{F}(c,d).
\end{eqnarray*}
So it is enough to determine the distribution of
$\mathbf{E}(c,d)$
and
$\mathbf{F}(c,d)$.
Since the determination is very technical and complex, we prefer to leave
it in Appendix B.
Using Lemmas~\ref{lemma:N_8_1}-\ref{lemma:N_8_3}, it is now a routine
check to see that
\begin{eqnarray*}
g( \N_8)&=&\left\{
  \begin{array}{ll}
    2^{3m+6}(3\cdot2^{m-1}-1),       &\uif g\in\E_4;\\
    2^{3m+6}(2^{m-1}-1),       &\uif g\in\E_5.
    \end{array}
  \right.
\end{eqnarray*}

\section{Conclusion}
In this paper, we construct a $9$-class scheme from a refinement of the
partition of the $\DG$ code by its Lee weights. We get the explicit
expressions of the dual scheme of this $9$-class scheme by manipulations
of complicated exponential sums and heavy computations. There is an interesting ``non-symmetry" between the characterization of the $9$-class scheme and its dual in the sense that the description of the original scheme reflects the properties of the underlying code while we see nothing about the code in the description of the dual scheme. Moreover, the dual scheme can not be described by the cosets of the Goethals codes as far as we see it. It will be interesting to see what code properties are reflected in the dual scheme.

In \cite{BD} Bonnecaze and Duursma showed that the partition of the
Kerdock (resp. shortened Kerdock) code induced by the complete weight
enumerators gives rise to an association scheme. Using this scheme and
the complete weight enumerator of the Kerdock (resp. shortened Kerdock)
code, they also showed that the complete weight enumerator of each coset
of the dual code, namely Preparata (resp. punctured Preparata) code can
be explicitly determined. Thus it is reasonable to believe that this is
also true for the $\DG$ code, which would give rise to a $22$-class
association scheme. Since the complete weight enumerator of the $\DG$
code has been explicitly determined by Shin, Kumar and Helleseth
(\cite{SKH}), once we figure out the parameters of this conjectured
scheme, then theoretically we know all about the complete weight
enumerators of each coset of the Goethals code. The coset weight
enumerators of the Goethals code have been studied by Helleseth and
Zinoviev (see~\cite{HZ1,HZ2}). We mention that our $9$-class scheme is
not a fusion scheme of this conjectured $22$-class scheme.

Davis and Xiang (see \cite{DX}) constructed the first known examples
where the non-homomorphic bijection approach outlined by Hagita and
Schmidt (see \cite{HS}) can produce negative Latin square type partial
difference sets in groups that previously had no known constructions.
The Cayley graphs of partial difference sets are strongly regular, so
yield  two-class association schemes.
Therefore it is interesting to investigate whether there are translation
schemes over elementary abelian $2$-groups with the same parameters as
those constructed in~\cite{ABS,BD} and this paper from the various $\Z_4$-linear codes.
When $m=3$, the Gray map image of the $\DG$ code is a $\Z_2$-linear code,
and we checked that the Gray map image of the $9$-class scheme remains a
scheme but with different parameters. However, when $m>3$, the binary $\DG$
code is no longer linear, and no scheme arises naturally in this way.

\section{Appendix A}
\renewcommand{\thesection}{A}
The first and second eigenmatrices of association scheme $\mathfrak{A}$:
\begin{scriptsize}
\renewcommand{\kbldelim}{(}
\renewcommand{\kbrdelim}{.}
$$
P=\kbordermatrix{
\mbox{}&\R_0&\R_1&\R_2\cup\R_3&\R_4\\
\E_0& 1&1/24\,{s}^{6}-1/8\,{s}^{4}+1/12\,{s
}^{2}&1/2\,{s}^{4}-4/3\,{s}^{2}+1/12\,{s}^{6}&-2+2\,{s}^{2}\\ 
\E_1\cup\E_9 &1&-1/12
\,{s}^{4}+1/12\,{s}^{2}&1/3\,{s}^{4}-4/3\,{s}^{2}&-2+2\,{s}^{2}\\
\E_2 &1&1/12\,{s}^{5}-1/4\,{s}^{3}+1/6\,s&1/2\,{
s}^{3}-4/3\,s+1/12\,{s}^{5}&0\\ 
\E_3\cup\E_7 &1&-1/6\,{s}^{3}
+1/6\,s&1/3\,{s}^{3}-4/3\,s&0\\ 
\E_4 &1&1/8\,{s}^{4}-1/4\,{s}^{2}&0&-2\\
\E_5 &1&-
1/4\,{s}^{2}&0&-2\\ 
\E_6 &1
&1/12\,{s}^{3}+1/6\,s&-4/3\,s-1/6\,{s}^{3}&0\\
\E_8&1&1/24\,{s}^{4}+1/12\,{s}^{2}&-4/3\,{s}^{2}-1/6\,
{s}^{4}&-2+2\,{s}^{2}\\
}
$$
\renewcommand{\kbldelim}{.}
\renewcommand{\kbrdelim}{)}
$$
\kbordermatrix{
\mbox{}&\R_5&\R_6\cup\R_7&\R_8&\R_9\\
&1/4\,{s}^{
6}-3/4\,{s}^{4}+1/2\,{s}^{2}&1/2\,{s}^{4}-4/3\,{s}^{2}+1/12\,{s}^{6}&1
/24\,{s}^{6}-1/8\,{s}^{4}+1/12\,{s}^{2}&1\\
&-1/2\,
{s}^{4}+1/2\,{s}^{2}&1/3\,{s}^{4}-4/3\,{s}^{2}&-1/12\,{s}^{4}+1/12\,{s
}^{2}&1\\
&0&-1/12\,{s}^{5}+4/3\,s-1/2\,{s}^{3}&-1/
12\,{s}^{5}+1/4\,{s}^{3}-1/6\,s&-1\\
&0&-1/3\,{s}^{3}+4/3\,s&1/6\,{s}^{3}-1/6
\,s&-1\\
&-1/4\,{s}^
{4}+1/2\,{s}^{2}&0&1/8\,{s}^{4}-1/4\,{s}^{2}&1\\
&1/2\,{s}^{2}&0&-1/4\,{s}^{2}&1\\
&0&1/6\,{s}^{3}+4/3\,s&-1/
12\,{s}^{3}-1/6\,s&-1\\
&1/4\,{s}^{4}+1/2\,{s}^{2}&-4/3\,{s}^{2}-1/6\,{s}
^{4}&1/24\,{s}^{4}+1/12\,{s}^{2}&1\\
},
$$
$$
Q= \left( \begin {array}{ccccc} 1&1/2\,{s}^{2}-1&{s}^{2}&1/2\, \left( {s
}^{2}-2 \right) {s}^{2}&1/2\,{s}^{2} \left( {s}^{2}-1 \right)
\\ \noalign{\medskip}1&-1&2\,s&-2\,s&3/2\,{s}^{2}\\ \noalign{\medskip}
1&1/2\,{s}^{2}-1&s&1/2\, \left( {s}^{2}-2 \right) s&0
\\ \noalign{\medskip}1&-1&s&-s&0\\ \noalign{\medskip}1&1/2\,{s}^{2}-1&0
&0&-1/2\,{s}^{2}\\ \noalign{\medskip}1&-1&0&0&-1/2\,{s}^{2}
\\ \noalign{\medskip}1&-1&-s&s&0\\ \noalign{\medskip}1&1/2\,{s}^{2}-1&
-s&-1/2\, \left( {s}^{2}-2 \right) s&0\end {array} \right.
$$
$$
\left. \begin {array}{ccc} 1/4\,{s}^{2} \left( {s}^{4}-3\,{s}^{2}+2
 \right) &1/6\,{s}^{2} \left( {s}^{4}-3\,{s}^{2}+2 \right) &1/6\,{s}^{
4}-1/2\,{s}^{2}+1/3\\ \noalign{\medskip}-3/2\,{s}^{2}&1/3\,s \left( {s
}^{2}+2 \right) &1/6\,{s}^{2}+1/3\\ \noalign{\medskip}0&-1/3\,s
 \left( {s}^{2}-1 \right) &1/3-1/3\,{s}^{2}\\ \noalign{\medskip}-1/4\,
 \left( {s}^{2}-2 \right) {s}^{2}&0&1/6\,{s}^{4}-1/2\,{s}^{2}+1/3
\\ \noalign{\medskip}1/2\,{s}^{2}&0&1/6\,{s}^{2}+1/3
\\ \noalign{\medskip}0&1/3\,s \left( {s}^{2}-1 \right) &1/3-1/3\,{s}^{
2}\\ \noalign{\medskip}-3/2\,{s}^{2}&-1/3\,s \left( {s}^{2}+2 \right)
&1/6\,{s}^{2}+1/3\\ \noalign{\medskip}1/4\,{s}^{2} \left( {s}^{4}-3\,{
s}^{2}+2 \right) &-1/6\,{s}^{2} \left( {s}^{4}-3\,{s}^{2}+2 \right) &1
/6\,{s}^{4}-1/2\,{s}^{2}+1/3\end {array} \right).
$$
\end{scriptsize}

The first and the second eigenmatrices of association
schemes $\mathfrak{B}$ and $\mathfrak{B}'$:
\renewcommand{\kbldelim}{(}
\renewcommand{\kbrdelim}{.}
\begin{scriptsize}
$$
P=\kbordermatrix{\mbox{}&\R_0&\R_1&\R_2&\R_3&\R_4\\
\E_0&1&1/24\,{s}^{6}-1/8\,{s}^{4}+1/12\,{s}^{
2}&1/2\,{s}^{4}-{s}^{2}&1/12\,{s}^{6}-1/3\,{s}^{2}&-2+2\,{s}^{2}
\\
\E_1& 1&-1/12\,{s}^{4}+1/12\,{s}^{2}&1/2\,{s}^{4}-{s}^{
2}&-1/6\,{s}^{4}-1/3\,{s}^{2}&-2+2\,{s}^{2}\\
\E_2& 1&1/12
\,{s}^{5}-1/4\,{s}^{3}+1/6\,s&1/2\,{s}^{3}-s&1/12\,{s}^{5}-1/3\,s&0
\\
\E_3& 1&-1/6\,{s}^{3}+1/6\,s&1/2\,{s}^{3}-s&-1/6\,{s}^{
3}-1/3\,s&0\\
\E_4& 1&1/8\,{s}^{4}-1/4\,{s}^{2}&0&0&-2
\\
\E_5& 1&-1/4\,{s}^{2}&0&0&-2\\
\E_6& 1&1/12
\,{s}^{3}+1/6\,s&-s&-1/6\,{s}^{3}-1/3\,s&0\\
\E_7& 1&-1/6
\,{s}^{3}+1/6\,s&-s&-1/3\,s+1/3\,{s}^{3}&0\\
\E_8& 1&1/24
\,{s}^{4}+1/12\,{s}^{2}&-{s}^{2}&-1/6\,{s}^{4}-1/3\,{s}^{2}&-2+2\,{s}^
{2}\\
\E_9& 1&-1/12\,{s}^{4}+1/12\,{s}^{2}&-{s}^{2}&-1/3\,
{s}^{2}+1/3\,{s}^{4}&-2+2\,{s}^{2}
}$$
\renewcommand{\kbldelim}{.}
\renewcommand{\kbrdelim}{)}
$$
\kbordermatrix{ \mbox{}&\R_5&\R_6&\R_7&\R_8&\R_9\\
&1/4\,{s}^{6}-3/4\,{s}^{4}+1/2\,{s}^{2}&1
/12\,{s}^{6}-1/3\,{s}^{2}&1/2\,{s}^{4}-{s}^{2}&1/24\,{s}^{6}-1/8\,{s}^
{4}+1/12\,{s}^{2}&1\\
&-1/2\,{s}^{4}+1/2\,{s}^{2}&-1/
6\,{s}^{4}-1/3\,{s}^{2}&1/2\,{s}^{4}-{s}^{2}&-1/12\,{s}^{4}+1/12\,{s}^
{2}&1\\
&0&-1/12\,{s}^{5}+1/3\,s&-1/2\,{s}^{3}+s&-1/
12\,{s}^{5}+1/4\,{s}^{3}-1/6\,s&-1\\ &0&1/6\,{s}^{3}+
1/3\,s&-1/2\,{s}^{3}+s&1/6\,{s}^{3}-1/6\,s&-1\\ &-1/4
\,{s}^{4}+1/2\,{s}^{2}&0&0&1/8\,{s}^{4}-1/4\,{s}^{2}&1
\\ &1/2\,{s}^{2}&0&0&-1/4\,{s}^{2}&1
\\ &0&1/6\,{s}^{3}+1/3\,s&s&-1/12\,{s}^{3}-1/6\,s&-1
\\ &0&1/3\,s-1/3\,{s}^{3}&s&1/6\,{s}^{3}-1/6\,s&-1
\\ &1/4\,{s}^{4}+1/2\,{s}^{2}&-1/6\,{s}^{4}-1/3\,{s}^
{2}&-{s}^{2}&1/24\,{s}^{4}+1/12\,{s}^{2}&1\\ &-1/2\,{
s}^{4}+1/2\,{s}^{2}&-1/3\,{s}^{2}+1/3\,{s}^{4}&-{s}^{2}&-1/12\,{s}^{4}
+1/12\,{s}^{2}&1
},$$

$$
Q=\left( \begin {array}{cccccc} 1&1/2\,{s}^{2}-1&{s}^{2}&1/2\, \left( {
s}^{2}-2 \right) {s}^{2}&1/2\,{s}^{2} \left( {s}^{2}-1 \right) &1/4\,{
s}^{2} \left( {s}^{4}-3\,{s}^{2}+2 \right) \\ \noalign{\medskip}1&-1&2
\,s&-2\,s&3/2\,{s}^{2}&-3/2\,{s}^{2}\\ \noalign{\medskip}1&1/2\,{s}^{2
}-1&s&1/2\, \left( {s}^{2}-2 \right) s&0&0\\ \noalign{\medskip}1&-1&s&
-s&0&0\\ \noalign{\medskip}1&1/2\,{s}^{2}-1&0&0&-1/2\,{s}^{2}&-1/4\,
 \left( {s}^{2}-2 \right) {s}^{2}\\ \noalign{\medskip}1&-1&0&0&-1/2\,{
s}^{2}&1/2\,{s}^{2}\\ \noalign{\medskip}1&-1&-s&s&0&0
\\ \noalign{\medskip}1&1/2\,{s}^{2}-1&-s&-1/2\, \left( {s}^{2}-2
 \right) s&0&0\\ \noalign{\medskip}1&-1&-2\,s&2\,s&3/2\,{s}^{2}&-3/2\,
{s}^{2}\\ \noalign{\medskip}1&1/2\,{s}^{2}-1&-{s}^{2}&-1/2\, \left( {s
}^{2}-2 \right) {s}^{2}&1/2\,{s}^{2} \left( {s}^{2}-1 \right) &1/4\,{s
}^{2} \left( {s}^{4}-3\,{s}^{2}+2 \right) \end {array} \right.
$$

$$
\left.\begin {array}{cccc} 1/6\,{s}^{2} \left( {s}^{4}-3\,{s}^{2}+2
 \right) &1/12\,{s}^{2} \left( {s}^{4}-4 \right) &1/6\,{s}^{4}-1/2\,{s
}^{2}+1/3&1/12\,{s}^{4}-1/3\\ \noalign{\medskip}1/3\,s \left( {s}^{2}+
2 \right) &-1/3\,s \left( {s}^{2}+2 \right) &1/6\,{s}^{2}+1/3&-1/6\,{s
}^{2}-1/3\\ \noalign{\medskip}-1/3\,s \left( {s}^{2}-1 \right) &-1/6\,
s \left( {s}^{2}+2 \right) &1/3-1/3\,{s}^{2}&-1/6\,{s}^{2}-1/3
\\ \noalign{\medskip}-1/3\,s \left( {s}^{2}-1 \right) &1/3\,s \left( {
s}^{2}-1 \right) &1/3-1/3\,{s}^{2}&-1/3+1/3\,{s}^{2}
\\ \noalign{\medskip}0&0&1/6\,{s}^{4}-1/2\,{s}^{2}+1/3&1/12\,{s}^{4}-1
/3\\ \noalign{\medskip}0&0&1/6\,{s}^{2}+1/3&-1/6\,{s}^{2}-1/3
\\ \noalign{\medskip}1/3\,s \left( {s}^{2}-1 \right) &-1/3\,s \left( {
s}^{2}-1 \right) &1/3-1/3\,{s}^{2}&-1/3+1/3\,{s}^{2}
\\ \noalign{\medskip}1/3\,s \left( {s}^{2}-1 \right) &1/6\,s \left( {s
}^{2}+2 \right) &1/3-1/3\,{s}^{2}&-1/6\,{s}^{2}-1/3
\\ \noalign{\medskip}-1/3\,s \left( {s}^{2}+2 \right) &1/3\,s \left( {
s}^{2}+2 \right) &1/6\,{s}^{2}+1/3&-1/6\,{s}^{2}-1/3
\\ \noalign{\medskip}-1/6\,{s}^{2} \left( {s}^{4}-3\,{s}^{2}+2
 \right) &-1/12\,{s}^{2} \left( {s}^{4}-4 \right) &1/6\,{s}^{4}-1/2\,{
s}^{2}+1/3&1/12\,{s}^{4}-1/3\end {array} \right).
$$
\end{scriptsize}

The first and the second eigenmatrices of association scheme
$\mathcal{C}$:
\begin{tiny}
$$
P= \left( \begin {array}{cccccc} 1&1/24\,{s}^{6}-1/8\,{s}^{4}+1/12\,{s}^
{2}&1/2\,{s}^{4}-{s}^{2}&1/12\,{s}^{6}-1/3\,{s}^{2}&{s}^{2}-1&1/8\,{s}
^{6}-3/8\,{s}^{4}+1/4\,{s}^{2}\\ \noalign{\medskip}1&-1/12\,{s}^{4}+1/
12\,{s}^{2}&1/2\,{s}^{4}-{s}^{2}&-1/6\,{s}^{4}-1/3\,{s}^{2}&{s}^{2}-1&
-1/4\,{s}^{4}+1/4\,{s}^{2}\\ \noalign{\medskip}1&1/8\,{s}^{4}-1/4\,{s}
^{2}&0&0&-1&-1/8\,{s}^{4}+1/4\,{s}^{2}\\ \noalign{\medskip}1&-1/4\,{s}
^{2}&0&0&-1&1/4\,{s}^{2}\\ \noalign{\medskip}1&1/24\,{s}^{4}+1/12\,{s}
^{2}&-{s}^{2}&-1/6\,{s}^{4}-1/3\,{s}^{2}&{s}^{2}-1&1/8\,{s}^{4}+1/4\,{
s}^{2}\\ \noalign{\medskip}1&-1/12\,{s}^{4}+1/12\,{s}^{2}&-{s}^{2}&-1/
3\,{s}^{2}+1/3\,{s}^{4}&{s}^{2}-1&-1/4\,{s}^{4}+1/4\,{s}^{2}
\end {array} \right),
$$
$$
Q= \left( \begin {array}{cccccc} 1&1/2\,{s}^{2}-1&1/2\,{s}^{2} \left( {s
}^{2}-1 \right) &1/4\,{s}^{2} \left( {s}^{4}-3\,{s}^{2}+2 \right) &1/6
\,{s}^{4}-1/2\,{s}^{2}+1/3&1/12\,{s}^{4}-1/3\\ \noalign{\medskip}1&-1&
3/2\,{s}^{2}&-3/2\,{s}^{2}&1/6\,{s}^{2}+1/3&-1/6\,{s}^{2}-1/3
\\ \noalign{\medskip}1&1/2\,{s}^{2}-1&0&0&1/3-1/3\,{s}^{2}&-1/6\,{s}^{
2}-1/3\\ \noalign{\medskip}1&-1&0&0&1/3-1/3\,{s}^{2}&-1/3+1/3\,{s}^{2}
\\ \noalign{\medskip}1&1/2\,{s}^{2}-1&-1/2\,{s}^{2}&-1/4\, \left( {s}^
{2}-2 \right) {s}^{2}&1/6\,{s}^{4}-1/2\,{s}^{2}+1/3&1/12\,{s}^{4}-1/3
\\ \noalign{\medskip}1&-1&-1/2\,{s}^{2}&1/2\,{s}^{2}&1/6\,{s}^{2}+1/3&
-1/6\,{s}^{2}-1/3\end {array} \right).
$$
\end{tiny}

The first and the second eigenmatrices of association scheme
$\mathcal{D}$:
\begin{tiny}
$$
P= \left( \begin {array}{ccccc} 1&1/24\,{s}^{6}-1/8\,{s}^{4}+1/12\,{s}^{
2}&1/2\,{s}^{4}-4/3\,{s}^{2}+1/12\,{s}^{6}&{s}^{2}-1&1/8\,{s}^{6}-3/8
\,{s}^{4}+1/4\,{s}^{2}\\ \noalign{\medskip}1&-1/12\,{s}^{4}+1/12\,{s}^
{2}&1/3\,{s}^{4}-4/3\,{s}^{2}&{s}^{2}-1&-1/4\,{s}^{4}+1/4\,{s}^{2}
\\ \noalign{\medskip}1&1/8\,{s}^{4}-1/4\,{s}^{2}&0&-1&-1/8\,{s}^{4}+1/
4\,{s}^{2}\\ \noalign{\medskip}1&-1/4\,{s}^{2}&0&-1&1/4\,{s}^{2}
\\ \noalign{\medskip}1&1/24\,{s}^{4}+1/12\,{s}^{2}&-4/3\,{s}^{2}-1/6\,
{s}^{4}&{s}^{2}-1&1/8\,{s}^{4}+1/4\,{s}^{2}\end {array} \right),
$$
$$
Q=\left( \begin {array}{ccccc} 1&1/2\,{s}^{2}-4/3+1/12\,{s}^{4}&1/2\,{s
}^{2} \left( {s}^{2}-1 \right) &1/4\,{s}^{2} \left( {s}^{4}-3\,{s}^{2}
+2 \right) &1/6\,{s}^{4}-1/2\,{s}^{2}+1/3\\ \noalign{\medskip}1&-4/3-1
/6\,{s}^{2}&3/2\,{s}^{2}&-3/2\,{s}^{2}&1/6\,{s}^{2}+1/3
\\ \noalign{\medskip}1&1/3\,{s}^{2}-4/3&0&0&1/3-1/3\,{s}^{2}
\\ \noalign{\medskip}1&1/2\,{s}^{2}-4/3+1/12\,{s}^{4}&-1/2\,{s}^{2}&-1
/4\, \left( {s}^{2}-2 \right) {s}^{2}&1/6\,{s}^{4}-1/2\,{s}^{2}+1/3
\\ \noalign{\medskip}1&-4/3-1/6\,{s}^{2}&-1/2\,{s}^{2}&1/2\,{s}^{2}&1/
6\,{s}^{2}+1/3\end {array} \right).
$$
\end{tiny}

\renewcommand{\thesection}{7}
\section{Appendix B}
\renewcommand{\thesection}{B}
\begin{lemma}\cite[Lemma 2]{HK}
  Let $\bfe=(e_X)_{X\in\Te}$ and let $E_j=\{X\mid e_X=j\}$ for $j=0,1,2,3.$
  The equation given by
  $$ \sum_{X\in\Te}e_XX=A+2B,\quad A,B\in\Te, e_X\in\Zfour$$
  is equivalent to the two binary equations
  $$ a=\sum_{X\in E_1\cup E_3}x$$ and
  $$ b^2=\sum_{X\in E_2\cup E_3}x^2+\sum_{X,Y\in E_1\cup E_3\atop{X<Y}}xy.$$
  \label{equivalency}
\end{lemma}

We set $2R=\{2x\mid x\in R\}$ and $-\Te=\{-X\mid X\in\Te\}$.

\begin{lemma}\cite[Theorem 1]{BD}
 Let $R=GR(4,m), m>0$, and $\Te$ be the Teichmuller set.
 \begin{enumerate}[(a)]
   \item The multiset $\Te+2\Te=\{X+2Y\mid X,Y\in\Te\}$ contains each element of $R$
         with multiplicity one.
   \item The multiset $\Te-\Te=\{X-Y\mid X,Y\in\Te\}$ contains $0$ with multiplicity $2^m$,
             no other elements of $2R$, and the elements outside $2R$ with multiplicity one.
   \item The multiset $\Te+\Te=\{X+Y\mid X,Y\in\Te\}$ contains the elements of $2R$ with
          multiplicity one, and half of the elements outside $2R$ with multiplicity two.
   \item The multiset $\Te+\Te$ and $-(\Te+\Te)$ coincide for $m$ even. For $m$ odd they intersect
         in $2R$. In particular, elements of $-\Te$ occur in $\Te+\Te$ only for $m$ even.
   \end{enumerate}
   \label{squareCount}
\end{lemma}

The next two lemmas are natural generalizations of Lemma \ref{squareCount}.

\begin{lemma}
 Let $R=GR(4,m), m$ odd, and $\Te$ be the Teichmuller set.
 Then
 \begin{enumerate}
   \item the multiset $\Te+\Te+\Te=\{X+Y+Z\mid X,Y,Z\in\Te\}$ contains
    each element of $-\Te$ with multiplicity one, and the elements outside $-\Te$ with multiplicity $2^m+1$;
   \item the multiset $\Te+\Te-\Te=\{X+Y-Z\mid X,Y,Z\in\Te\}$ contains
   each element of $\Te$ with multiplicity $2^{m+1}-1$, and the elements outside $\Te$ with multiplicity $2^m-1$.
\end{enumerate}
\label{cubicCount}
\end{lemma}
\begin{proof}
Here we only give the proof of $(1)$, since a similar argument
will prove $(2)$.
We need to investigate the number of solutions of the following equation
\[ X+Y+Z=C, \] where $X,Y,Z\in\Te$ and $C\in R$.
The element $C$ can be uniquely expressed as $C=A+2B$, where $A,B\in\Te$.

First if $C\in-\Te$ i.e., $A=B$, then we have $X+Y+Z+A=0$, implying $X=Y=Z=A$.
So the multiset $\Te+\Te+\Te$ contains each element of $-\Te$
with multiplicity one.

Next if $C\notin -\Te,$ i.e., $A\neq B$, we split our consideration into two
cases $A=0$ and $A\neq 0$.

When $A=0$, without loss of generality, we consider the equation
\begin{equation}
X+Y+Z=2,
\label{X+Y+Z=2}
\end{equation}
which is equivalent to the following system of binary equations
\begin{eqnarray}
  \left\{
  \begin{array}{lllllll}
    x&+&y&+&z&=& 0,\\
    xy&+&xz&+&yz&=& 1,\\
  \end{array}
  \right.
  \label{sys_X+Y+Z=2}
\end{eqnarray}
by Lemma \ref{equivalency}. From this system, we have
$$ x^2+y^2+xy+1=0. $$
Assume that $x=0,$ we get one solution $$(x,y,z)=(0,1,1).$$
Assume that $x\neq 0$ and $y=tx$ for some $t\in \Fpm$, then we obtain
$$ (t^2+t+1)x^2+1=0, $$
implying
\[(x,y,z)=\left(\frac{1}{\sqrt{t^2+t+1}},\frac{t}{\sqrt{t^2+t+1}},\frac{t+1}{\sqrt{t^2+t+1}}\right).\]
Hence there are totally $2^m+1$ solutions for Equation (\ref{X+Y+Z=2}).

When $A\neq0$, without loss of generality, we only need to investigate the equation
\begin{equation}
 X+Y+Z=1+2B,
  \label{X+Y+Z=1+2B}
\end{equation}
where $B\in\Te$ and $B\neq1$.
By Lemma \ref{equivalency}, Equation (\ref{X+Y+Z=1+2B}) is equivalent to the following system
\begin{eqnarray}
  \left\{
  \begin{array}{lllllll}
    x&+&y&+&z&=& 1,\\
    xy&+&xz&+&yz&=& b^2,\\
  \end{array}
  \right.
  \label{sys_X+Y+Z=1+2B}
\end{eqnarray}
from which we have
\begin{equation}
 x^2+xy+y^2+x+y+b^2=0.
  \label{app_X+Y+Z=1+2B}
\end{equation}
Now let $x=(1+b)u+b$ and $y=(1+b)v+b$. Then the above equation becomes
$$
u^2+uv+v^2+u+v=0,
$$
so without loss of generality, we can assume $b=0$ in Equation (\ref{app_X+Y+Z=1+2B}).

Assume that $x=0,$ then we get two solutions $(x,y,z)=(0,0,1)$ and $(x,y,z)=(0,1,0)$.

Assume that $x\neq 0$ and $y=tx,$ for some $t\in \Fpm$, then we have
$$ (t^2+t+1)x^2+(t+1)x=0$$
implying $$(x,y,z)=(0,0,1)$$ or
$$(x,y,z)=\left(\frac{t+1}{t^2+t+1},\frac{t^2+t}{t^2+t+1},\frac{t}{t^2+t+1}\right).$$
If $t=1,$ there is only one solution $(x,y,z)=(0,0,1)$.
Hence there are totally $2^m+1$ solutions for Equation (\ref{app_X+Y+Z=1+2B}).
\end{proof}\qed

The above lemma implies the following result immediately.

\begin{corollary}\label{coro: quaCount}
 Let $R=GR(4,m), m$ odd, and $\Te$ be the Teichmuller set.
 \begin{enumerate}[(a)]
   \item The multiset $\Te+\Te+\Te+\Te=\{X+Y+Z+W:X,Y,Z,W\in\Te\}$ contains $0$ with multiplicity $2^m$,
     the elements of $2R\backslash\{0\}$ with multiplicity $2^m(2^m+1)$, and the elements outside $2R$
     with multiplicity $2^{2m}$.
   \item The multiset $\Te+\Te-\Te-\Te=\{X+Y-Z-W:X,Y,Z,W\in\Te\}$ contains $0$ with multiplicity $(2^{m+1}-1)2^m$,
     the elements of $2R\backslash\{0\}$ with multiplicity $(2^m-1)2^m$, and the elements outside $2R$
     with multiplicity $2^{2m}$.
   \item The multiset $\Te+\Te+\Te-\Te=\{X+Y+Z-W:X,Y,Z,W\in\Te\}$ contains the elements of $2R$ with multiplicity $2^{2m}$,
     the elements of $(\Te+\Te)\backslash 2R$ with multiplicity $(2^m+1)2^m$, and the elements outside $\Te+\Te$ with multiplicity $(2^m-1)2^m$.
 \end{enumerate}
\end{corollary}

Let $f_a(x)=x^3+x+a$ and set
$$M_i=\left\{ a\in\Fpm,a\neq 0\mid f_a(x)=0\quad\textup{has precisely $i$ solutions in $\Fpm$} \right\}$$
for $i=0,1,3.$ The exact values of the three numbers $|M_0|,|M_1|,|M_3|$ have been computed in
the appendix of \cite{KHCH}:
\begin{eqnarray*}
  |M_0|&=& \frac{q+1}{3},\\
  |M_1|&=& \frac{q}{2}-1,\quad\textup{and}\\
  |M_3|&=& \frac{q-2}{6}.
\end{eqnarray*}


\begin{lemma}\label{x3+x}
  Let $f_a(x)=x^3+x+a$ for some $a\in\Fpm.$
  \begin{enumerate}
    \item[(1)] If $a=0$, then $f_a(x)$ has two zeroes $x=0$ and $1$ in $\Fpm$.
    \item[(2)] If $a=b+b^{-1}$ for some $b\in\Fpm\backslash\{0,1\}$, then $f_a(x)$
          has one and only one zero in $\Fpm$.
    \item[(3)] If $a=b^{-1}+b^{-3}$ for some $b\in\Fpm\backslash\{0,1\}$ satisfying $\tr(b)=1$,
              then $f_a(x)$ has three distinct zeroes in $\Fpm$.
    \item[(4)] If $a$ satisfies none of the above conditions, then $f_a(x)$ is irreducible
          over $\Fpm$.
  \end{enumerate}
  \label{cubicSolutions}
\end{lemma}
\begin{proof}
  (1) is immediate. One finds in the literature (see \cite[p.169]{B}) that
  $f_a(x)=0$ has a unique solution in $\Fpm$ if and only if $\tr(1/a)=0.$
  If $a=b+b^{-1}$ for some $b\in\Fpm\backslash\{0,1\}$, then
  $$\tr\left(\frac{1}{b+b^{-1}}\right)=\tr\left(\frac{b}{b^2+1}\right)=
  \tr\left(\frac{b}{b+1}+\left(\frac{b}{b+1}\right)^2\right)=0.$$
  It is easy to check that the size of the set $\{b+b^{-1}\mid b\in\Fpm\backslash\{0,1\}\}$
  is equal to $|M_1|$. Then (2) follows.

  If $a=b^{-1}+b^{-3}$ for some $b\in\Fpm\backslash\{0,1\}$ satisfying $\tr(b)=1$, then we have
  $$\tr\left(\frac{1}{b^{-1}+b^{-3}}\right)=\tr\left( b+\frac{b}{b^2+1} \right)=\tr(b)=1.$$
  Since $b^{-1}$ is already a zero of $f_a(x)$, we see that $f_a(x)$ must has three distinct zeroes
  in $\Fpm$. The next thing is to check that the cardinality of the set
  $\{b^{-1}+b^{-3}\mid b\in\Fpm\backslash\{0,1\},\tr(b)=1\}$ is the same as $|M_3|$.
  Suppose that $b^{-1}+b^{-3}=c^{-1}+c^{-3}$ for some $b\in\Fpm\backslash\{0,1\}$ satisfying $\tr(b)=1$
  and $c\in\Fpm\backslash\{0,1\}$. We see that
  \begin{eqnarray*}
    &b^{-1}+b^{-3}=c^{-1}+c^{-3}\\
  \Longleftrightarrow&(b+c)(b^2c^2+b^2+bc+c^2)=0\\
  \Longleftrightarrow& c=b \textup{ or } c^2+(b/(b^2+1))c+b^2/(b^2+1)=0.\\
  \end{eqnarray*}
  Because
  $$\tr\left( \frac{b^2/(b^2+1)}{b^2/(b^2+1)^2} \right)=\tr(b^2+1)=0,$$
  the equation $c^2+(b/(b^2+1))c+b^2/(b^2+1)=0$ has two distinct zeroes in $\Fpm.$
  From the symmetry of the equation $b^2c^2+b^2+bc+c^2=0$, the value $\tr(c)$ must equal to $1$.
  Now it follows that the size of the set
  $\{b^{-1}+b^{-3}\mid b\in\Fpm\backslash\{0,1\},\tr(b)=1\}$ is $(q-2)/6=|M_3|$, finishing
  the prove of (3). Then (4) is immediate.
\end{proof}\qed

\begin{lemma}
Let $W\in\Te$. The number of solutions to the following system of equations
$$
\left\{
  \begin{array}{lllllllll}
  X&+&Y&-&Z&=&W&+&2;\\
  x^3&+&y^3&+&z^3&=&w^3&+&d,
  \end{array}
\right.
$$
is
(1) $1$, when $d=0$; (2) $2$, when $d\in M_1$; (3) $0$, when $d\in M_0\cup M_3$.
\end{lemma}
\begin{proof}
Using the fact that $X+Y=(\sqrt{X}+\sqrt{Y})^2+2\sqrt{XY}$,
the first equation $X+Y=2+Z+W$ translates to the following
two equations over $\Fpm$:
\[
x+y=z+w,\quad xy=zw+1.
\]
Now we compute
\begin{align*}
d&=x^3+y^3+z^3+w^3\\
&=(x+y)^3+xy(x+y)+z^3+w^3\\
&=(z+w)^3+(zw+1)(z+w)+z^3+w^3\\
& =z+w.
\end{align*}
Hence $z=d+w$ and $y=d+x$. Plugging them into $xy=zw+1$,
we get $(x+z)^2+d(x+z)+1=0.$
When $d=0$, we get $x=1+z=d+w+1$, so this system only has $1$ solution.
When $d\ne 0$, we have
$(x+z)^2/d^2+(x+z)/{d}+{1}/{d^2}=0$.
This equation has $0$ or $2$ solutions depending on
whether $\tr({1}/{d})=1$ or not.
\end{proof}\qed

\begin{lemma}\label{lemma:N_6_1}
Let $A, B\in\Te$ and $B\neq 0$.
The number of solutions to the following system of equations
$$
\left\{
  \begin{array}{lllllllll}
  X&+&Y&-&Z&=&A&+&2B;\\
  x^3&+&y^3&+&z^3&=&a^3&+&b^3d,
  \end{array}
\right.
$$
is
(1) $1$, when $d=0$; (2) $2$, when $d\in M_1$; (3) $0$, when $d\in M_0\cup M_3$.
\end{lemma}

Similarly, we get the following result.

\begin{lemma}\label{lemma:N_6_2}
Let $A, B\in\Te$ and $B\neq 0$.
The number of solutions to the following system
$$
\left\{
  \begin{array}{lllllllll}
    X&+&Y&+&Z&=&-A&+&2B,\\
    x^3&+&y^3&+&z^3&=&a^3&+&b^3d,
  \end{array}
  \right.
$$
is (1) $3$, when $d=0$; (2) $0$, when $d\in M_0\cup M_1$; (3) $6$, when $d\in M_3$.
\end{lemma}

The following corollaries can be easily verified from the proof of the above lemmas.

\begin{corollary}\label{coro: N_8_1}
Let $0\neq B\in\Te.$ The number of solutions to the following system of equations
$$
\left\{
  \begin{array}{lllllllll}
  X&+&Y&-&Z&-W&=&2B;\\
  x^3&+&y^3&+&z^3+&w^3&=&b^3d,
  \end{array}
\right.
$$
is
(1) $2^m$, when $d=0$; (2) $2^{m+1}$, when $d\in M_1$; (3) $0$, when $d\in M_0\cup M_3$.
\end{corollary}

\begin{corollary}\label{coro: N_8_2}
Let $0\neq B\in\Te.$ The number of solutions to the following system
$$
\left\{
  \begin{array}{lllllllll}
X&+&Y&+&Z&+&W&=& 2B;\\
x^3&+&y^3&+&z^3&+&w^3&=&b^3d,
\end{array}
\right.
$$
is
(1) $3\cdot 2^m$, when $d=0$; (2) $0$, when $d\in M_0\cup M_1$; (3) $6\cdot 2^m$,  when $d\in M_3$.
\end{corollary}

\renewcommand{\thesection}{8}
\section{Appendix C}
\renewcommand{\thesection}{C}
For $(a,b)\in G\times \Fq$, the sum $\xi(a,b)$ has the following properties:
\begin{enumerate}
  \item If $b\neq0,$ then $\xi(a,b)=\xi(aB^{-1/3},1)$;
  \item If $U,V\in\Te,W\in\Tem$, then $\xi(U+2V,b)=\xi(UW+2VW,bw^3)$.
\end{enumerate}

Now we fix some notation. Let $a$ and $c$ be elements of $R$,
and write them as $a=U+2V$ and $c=S+2T$ where $U,V,S,T\in\Te.$
For convenience, we define $\eta_a$ as $\eta_a=\xi(a,1).$
Let $u$ be the projection of $a$ modulo $2$ in $\Fq$.
Set
$$f_{u}(z)=z^2+u^2z+\sqrt{z}+u$$
and $F_u$ be the zeros of $f_{u}(z)$ in $\Fq.$
Also set
$$h_{u}(z)=f_u(z)-u=z^2+u^2z+\sqrt{z}$$
and $H_u$ be the zeros of $h_u(z)$ in $\Fq.$
It is easy to see $u^2\in F_u$ and
$$F_u=\{x+u^2\mid x\in H_u\},$$
so $|F_u|=|H_u|$.
For each $x\in H_u$, we have $\tr(ux)=\tr(u^2x^2)=\tr(x^3+x^{3/2})=0$, which yields that $\tr(uy)=\tr(u^3)$ for each $y\in F_u$.
For $X,Y\in\Te,$ we see that
$$X+Y=(\sqrt{X}+\sqrt{Y})^2+2\sqrt{XY}.$$
The element $(\sqrt{X}+\sqrt{Y})^2\in\Te$
will be denoted as $X\oplus Y$ in the following.
Here come some useful results involving $\eta_a$.

\begin{lemma}\label{lemma:eta}
Let $a\in R$ and $u$ be the projection of $a$ modulo $2$ in $\Fq$.
The exponential sum $\eta_a$ satisfies:
  \begin{align*}
     \eta_a^2&=2^m\sum_{Z\in\Te\atop{f_u(z)=0}} i^{\T(aZ+2Z^3)},&&
     \eta_a\overline{\eta_a}=2^m\sum_{Z\in\Te\atop{h_u(z)=0}} i^{\T(aZ+2Z^3)},&&\\
     \eta_a^4&=2^{m}(-1)^{\tr(u^3)}|F_u|\eta_a\overline{\eta_a},&&
     (\eta_a\overline{\eta_a})^2= 2^{m}|F_u|\eta_a\overline{\eta_a},&&
     \eta_a^3\overline{\eta_a}= 2^{m}|F_u|\eta_a^2.
  \end{align*}
\end{lemma}
\begin{proof}
We first compute
\begin{eqnarray*}
  \eta_a^2&=&\sx\sy i^{\T(a(X+Y)+2(X^3+Y^3))}
  = \sy\sz i^{\T(a(Y\oplus Z+Y)+2((Y\oplus Z)^3+Y^3))}\\
  &=& \sy\sz i^{T(aZ+2Z^3+2(aY+a\sqrt{YZ}+Y^2Z+Z^2Y+Z^3))}\\
  &=& \sz i^{\T(aZ+2Z^3)}\sum_{y\in\Fq} (-1)^{\tr(y(z^2+u^2z+\sqrt{z}+u))}\\
  &=& 2^m\,\sum_{Z\in\Te\atop{f_u(z)=0}} i^{\T(aZ+2Z^3)}.
\end{eqnarray*}
A similar analysis would give
$$\eta_a\overline{\eta_a}=2^m\sum_{Z\in\Te\atop{h_u(z)=0}} i^{\T(aZ+2Z^3)}.$$
Then
\begin{eqnarray*}
  \eta_a^4&=& 2^{2m}\sum_{Z\in\Te\atop{f_u(z)=0}}\sum_{W\in\Te\atop{f_u(w)=0}}i^{\T(a(Z+W)+2(Z^3+W^3))}\\
  &=& 2^{2m}\sum_{Z\in\Te\atop{f_u(z)=0}}\sum_{W\in\Te\atop{h_u(w)=0}}i^{\T(a(Z+Z\oplus W)+2(Z^3+(Z\oplus W)^3))}\\
  &=&2^{2m}\sum_{W\in\Te\atop{h_u(w)=0}}i^{\T(aW+2W^3)}\sum_{Z\in\Te\atop{f_u(z)=0}}(-1)^{\tr(uz+w(z^2+u^2z+\sqrt{z}))}\\
  &=&2^{2m}\sum_{W\in\Te\atop{h_u(w)=0}}i^{\T(aW+2W^3)}\sum_{Z\in\Te\atop{f_u(z)=0}}(-1)^{\tr(u(z+w))}\\
  &=& 2^{m}(-1)^{\tr(u^3)}|F_u|\eta_a\overline{\eta_a}.
\end{eqnarray*}
The same calculations as above will prove the rest of equations.
\end{proof}\qed

\begin{lemma}\label{lemma:N_8_1}
Let $c\in R\backslash 2R$ and $d\in\Fq$. Then $c$ can be
expressed uniquely as $c=F-G$ where $F,G\in\Te$. We have
\begin{eqnarray*}
  \mathbf{E}(c,d)
              &=& \left\{
                \begin{array}{ll}
                  2^{3m+4}(3\cdot2^{m-1}-1), &\uif d=f^3+g^3;\\
                  2^{3m+4}(2^{m-1}-1), &\uif d\neq f^3+g^3,
                \end{array}
                \right.
\end{eqnarray*}
where $f,g$ are the projections of $F,G$ modulo $2$ in $\Fq.$
\end{lemma}
\begin{proof}
It is direct to check that $$st^2+s^{-1}t^4+s^3=f^3+g^3$$ and $$\tr(s^{-3}(f^3+g^3))=1.$$

Let $X\in\Te$ and $B\in\Tem$. We first compute that
\begin{eqnarray*}
\mathcal{U}&:=&i^{\T(Xc)}\sum_{Y\in\Te}\eta_{(X+2Y)B^{-1/3}}\overline{\eta_{(X+2Y)B^{-1/3}}}(-1)^{\tr(ys)}\\
&=&2^mi^{\T(Xc)}\sum_{Y\in\Te}\sum_{Z\in\Te\atop{h_{x'}(z)=0}}i^{\T((X+2Y)B^{-1/3})Z+2Z^3)}(-1)^{\tr(ys)}\\
&=&2^mi^{\T(Xc)}\sum_{Z\in\Te\atop{h_{x'}(z)=0}}i^{\T(XB^{-1/3}Z+2Z^3)}\sum_{Y\in\Te}(-1)^{\tr(y(b^{-1/3}z+s))}\\
&=&\left\{
  \begin{array}{ll}
    2^{2m}(-1)^{\tr(b(f^3+g^3))},&\uif X=B^{\frac{1}{2}}S^{\frac{1}{2}}\oplus B^{\frac{1}{4}}S^{-\frac{1}{4}};\\
0,& \textup{ otherwise,}
  \end{array}
  \right.
\end{eqnarray*}
where $x'=xb^{-1/3}$.
Notice that $B^{\frac{1}{2}}S^{\frac{1}{2}}\oplus B^{\frac{1}{4}}S^{-\frac{1}{4}}=0$ if $b=s^{-3}$.

Let $x,s\in\Fq^*$ and $b\in\Fq^*\backslash\{s^{-3}\}$.
We consider the set $H_x$. Immediately $0$ is a root of $h_x(z)$.
The equation $h_x(z)=0$ is equivalent to $$(h_x(z))^2=z(z^3+x^4z+1)=0.$$
Replace $z$ with $x^2w$, then $z^3+x^4z+1=0$ becomes $w^3+w+x^{-6}=0$.\\
If $$x=(bs)^{1/2}+(bs^{-1})^{1/4},$$ then
$$x^6=b^{1/2}s^{3/2}+bs^3+b^{-1/2}s^{-3/2}+1$$
and
$$x^{-6}=(1+b^{-1/2}s^{-3/2})^{-1}+(1+b^{-1/2}s^{-3/2})^{-3}.$$
Using Lemma \ref{cubicSolutions}, we see that if
$$\tr(1+b^{-1/2}s^{-3/2})=0,$$
then $$\tr(x^6)=\tr(1+b^{-1}s^{-3})=0,$$
implying
$$|F_{b^{\frac{1}{6}}s^{\frac{1}{2}}+b^{-\frac{1}{12}}s^{-\frac{1}{4}}}|=|H_{b^{\frac{1}{6}}s^{\frac{1}{2}}+b^{-\frac{1}{12}}s^{-\frac{1}{4}}}|=2;$$
otherwise we have
$$|F_{b^{\frac{1}{6}}s^{\frac{1}{2}}+b^{-\frac{1}{12}}s^{-\frac{1}{4}}}|=|H_{b^{\frac{1}{6}}s^{\frac{1}{2}}+b^{-\frac{1}{12}}s^{-\frac{1}{4}}}|=4.$$
So the expression
$$(2(-1)^{\tr(1+b^{-1}s^{-3})}+6)|F_{b^{\frac{1}{6}}s^{\frac{1}{2}}+b^{-\frac{1}{12}}s^{-\frac{1}{4}}}|$$
is always equal to 16.

The sum $\mathbf{E}(c,d)$ will be divided into four parts, and computed
separately with the assistance of Lemma~\ref{lemma:eta}.
We first compute
\begin{eqnarray*}
\mathbf{E}(c,d)_1
&=&\sum_{a\in R^*}\sum_{b\in\Fq^*}\left(\xi^4(a,b)+\overline{\xi^4(a,b)}
+6\,\xi^2(a,b)\overline{\xi^2(a,b)}\right)i^{\T(ac+2bd)}\\
&=&\sum_{X\in\Tem}\sum_{Y\in\Te}\sum_{b\in\Fq^*}\left(\xi^4(X+2Y,b)
    +\overline{\xi^4(X+2Y,b)}\right.\\
    &&\qquad\qquad\qquad\qquad\qquad\left.+6\xi^2(X+2Y,b)\overline{\xi^2(X+2Y,b)}\right)i^{\T((X+2Y)c+2bd)}\\
&=&\sum_{X\in\Tem}\sum_{Y\in\Te}\sum_{b\in\Fq^*}\left(\eta^4_{(X+2Y)B^{-\frac{1}{3}}}+\overline{\eta^4_{(X+2Y)B^{-\frac{1}{3}}}}\right.\\
    &&\qquad\qquad\qquad\qquad\qquad\left.+6\eta^2_{(X+2Y)B^{-\frac{1}{3}}}\overline{\eta^2_{(X+2Y)B^{-\frac{1}{3}}}}\right)i^{\T((X+2Y)c+2bd)}\\
&=&2^m\sum_{X\in\Tem}\sum_{b\in\Fq^*}(-1)^{\tr(bd)}\left(2(-1)^{\tr(x^3b^{-1})}+6\right)|F_{xb^{-\frac{1}{3}}}|\,\mathcal{U}\\
&=&2^{3m}\sum_{b\in\Fq^*\backslash\{s^{-3}\}}\left[\left(2(-1)^{\tr(1+b^{-1}s^{-3})}+6\right)|F_{b^{\frac{1}{6}}s^{\frac{1}{2}}+b^{-\frac{1}{12}}s^{-\frac{1}{4}}}|\right]
(-1)^{\tr(b(f^3+g^3+d))}\\
&=&2^{3m+4}\sum_{b\in\Fq^*\backslash\{s^{-3}\}}(-1)^{\tr(b(f^3+g^3+d)).}\\
\end{eqnarray*}
Next,
\begin{eqnarray*}
\mathbf{E}(c,d)_2&=&\sum_{X\in\Tem}\sum_{Y\in\Te}\left(\xi^4(X+2Y,0)+\overline{\xi^4(X+2Y,0)}\right.\\
&&\qquad\qquad\qquad\qquad\qquad\left.+6\xi^2(X+2Y,0)\overline{\xi^2(X+2Y,0)}\right)i^{\T((X+2Y)c)}\\
&=&\sum_{X\in\Tem}\sum_{Y\in\Te}\left(2^{2m}i^{\T(2)}+2^{2m}i^{-\T(2)}+6\cdot2^{2m}\right)i^{\T((X+2Y)c)}\\
&=& 2^{2m+2}\sum_{X\in\Tem}i^{\T(Xc)}\sum_{Y\in\Te}(-1)^{\tr(ys)}\\
&=&0.
\end{eqnarray*}
Then,
\begin{eqnarray*}
\mathbf{E}(c,d)_3&=&\sum_{Y\in\Te}\sum_{b\in\Fq^*}\left(\xi^4(2Y,b)+\overline{\xi^4(2Y,b)}+6\xi^2(2Y,b)\overline{\xi^2(2Y,b)}\right)i^{\T(2Yc+2bd)}\\
&=&\sum_{Y\in\Te}\sum_{b\in\Fq^*}\left(2^{2m+4}(1+(-1)^{\tr(yb^{-1/3}+1)}\right)(-1)^{\tr(ys+bd)}\\
&=&2^{2m+4}\sum_{b\in\Fq^*}(-1)^{\tr(bd+1)}\sum_{Y\in\Te}(-1)^{\tr(y(b^{-1/3}+s))}\\
&=& 2^{3m+4}(-1)^{\tr(s^{-3}d+1)}.
\end{eqnarray*}
At last,
\begin{eqnarray*}
\mathbf{E}(c,d)_4&=&\xi^4(0,0)+\overline{\xi^4(0,0)}+6\xi^2(0,0)\overline{\xi^2(0,0)}=2^{4m+3}.
\end{eqnarray*}
Adding $\mathbf{E}(c,d)_1,\mathbf{E}(c,d)_2,\mathbf{E}(c,d)_3$,
and $\mathbf{E}(c,d)_4$ up will complete the proof.
\end{proof}\qed

\begin{lemma}\label{lemma:N_8_2}
Let $c\in R\backslash 2R$ and $d\in\Fq.$ There exist $F,G\in\Te$ such that
$c=F+G$ with $F\neq G.$ We have
\begin{eqnarray*}
  \mathbf{F}(c,d)
              &=& \left\{
                \begin{array}{ll}
                  2^{3m+2}(3\cdot2^{m-1}-1), &\uif d=f^3+g^3;\\
                  2^{3m+2}(2^{m-1}-1),      &\uif d\neq f^3+g^3,\\
                \end{array}
                \right.
\end{eqnarray*}
where $f,g$ are the projections of $F,G$ modulo $2$ in $\Fq.$
\end{lemma}
\begin{proof}
It is easy to check that $$st^2+s^3=f^3+g^3.$$
Let $X\in\Te$ and $B\in\Tem$. We see
\begin{eqnarray*}
\mathcal{V}&:=&i^{\T(X)}\sum_{Y\in\Te}\left(\eta_{(X+2Y)B^{-1/3}}^3\overline{\eta_{(X+2Y)B^{-1/3}}}\right.\\
&&\qquad\qquad\qquad\qquad\qquad\left.+\eta_{(X+2Y)B^{-1/3}}\overline{\eta_{(X+2Y)B^{-1/3}}^3)}\right)(-1)^{\tr(y)}\\
&=&2^{2m}|K_{xb^{-1/3}}|i^{\T(X)}\sum_{Z\in\Te\atop{f_{x'}(z)=0}}\left( i^{\T(XB^{-1/3}Z+2Z^3)}\right.\\
&&\qquad\qquad\qquad\qquad\qquad\left.+i^{-\T(XB^{-1/3}Z+2Z^3)} \right)\sum_{Y\in\Te}(-1)^{\tr(v(b^{-1/3}z+1))}\\
&=&\left\{
  \begin{array}{ll}
2^{3m}|K_{b^{1/6}}|(1+(-1)^{\tr(b)}),&\uif U=B^{1/2};\\
2^{3m}|K_{b^{1/6}+b^{-1/3}}|(-1+(-1)^{\tr(b)}),&\uif U=1\oplus B^{1/2};\\
0,& \textup{ otherwise},
  \end{array}
  \right.
\end{eqnarray*}
where $x'=xb^{-1/3}$.

We first calculate
\begin{eqnarray*}
\mathbf{F}(c,d)_1&=&\sum_{X\in\Tem}\sum_{Y\in\Te}\sum_{b\in\Fq^*}\left(\xi^3(X+2Y,b)\overline{\xi(X+2Y,b)}\right.\\
&&\qquad\qquad\qquad\qquad\qquad\left.+\xi(X+2Y,b)\overline{\xi^3(X+2Y,b)}\right)i^{\T((X+2Y)c+2bd)}\\
&=&\sum_{X\in\Tem}\sum_{Y\in\Te}\sum_{b\in\Fq^*}\left(\xi^3((X+2Y)S,bs^3)\overline{\xi((X+2Y)S,bs^3)}\right.\\
&&\qquad\left.+\xi((X+2Y)S,bs^3)\overline{\xi^3((X+2Y)S,bs^3)}\right)i^{\T((X+2Y)S(1+2S^{-1}T)+2bd)}\\
&=&\sum_{X\in\Tem}\sum_{Y\in\Te}\sum_{b\in\Fq^*}\left(\xi^3(X+2Y,b)\overline{\xi(X+2Y,b)}\right.\\
&&\qquad\qquad\qquad\left.+\xi(X+2Y,b)\overline{\xi^3(X+2Y,b)}\right)i^{\T((X+2Y)(1+2S^{-1}T)+2bs^{-3}d)}\\
&=&\sum_{X\in\Tem}\sum_{b\in\Fq^*}i^{\T(2XS^{-1}T+2bs^{-3}d)}\,\mathcal{V}\\
&=&2^{3m}\sum_{b\in\Fq^*}|K_{b^{1/6}}|(-1)^{\tr(b^{1/2}s^{-1}t+bs^{-3}d)}(1+(-1)^{\tr(b)})\\
&&+2^{3m}\sum_{b\in\Fq\backslash\{0,1\}}|K_{b^{1/6}+b^{-1/3}}|(-1)^{\tr(b^{1/2}s^{-1}t+s^{-1}t+bs^{-3}d)}(-1+(-1)^{\tr(b)})\\
&=&2^{3m+2}\left(\sum_{b\in\Fq^*\atop{\tr(b)=0}}(-1)^{\tr(b(s^{-2}t^2+s^{-3}d))}+
\sum_{b\in\Fq\backslash\{0,1\}\atop{\tr(b)=1}}(-1)^{\tr(s^{-1}t+b(s^{-2}t^2+s^{-3}d))}\right)\\
&=&2^{3m+2}(1+(-1)^{\tr(s^{-3}d+1)})\sum_{b\in\Fq^*\atop{\tr(b)=0}}(-1)^{\tr(b(s^{-2}t^2+s^{-3}d))}\\
&=&\left\{
  \begin{array}{ll}
    2^{3m+3}(2^{m-1}-1), &\uif d=f^3+g^3;\\
    -2^{3m+2}(1+(-1)^{\tr(s^{-3}d+1)}),  &\textup{ otherwise}.
  \end{array}\right.
\end{eqnarray*}
Next,
\begin{eqnarray*}
\mathbf{F}(c,d)_2&=&\sum_{X\in\Tem}\sum_{Y\in\Te}\left(\xi^3(X+2Y,0)\overline{\xi(X+2Y,0)}\right.\\
&&\qquad\qquad\qquad\qquad\qquad\left.+\xi(X+2Y,0)\overline{\xi^3(X+2Y,0)}\right)i^{\T((X+2Y)c)}\\
&=&2^{2m}\sum_{X\in\Tem}\sum_{Y\in\Te}\left(i^{\T(1+2X^{-1}Y+(X+2Y)c)}+i^{\T(-1-2X^{-1}Y+(X+2Y)c)}\right)\\
&=& 2^{2m}\left(\sum_{X\in\Tem}i^{\T(1+Xc)}+\sum_{X\in\Tem}i^{\T(-1+Xc)}\right)\sum_{Y\in\Te}(-1)^{\tr(y(x^{-1}+s))}\\
&=& 2^{3m}\left( i^{\T(2+2s^{-1}t)}+i^{\T(2s^{-1}t)} \right)\\
&=& 0.
\end{eqnarray*}
Then we calculate
\begin{eqnarray*}
\mathbf{F}(c,d)_3&=&\sum_{Y\in\Te}\sum_{b\in\Fq^*}(\xi^3(2Y,b)\overline{\xi(2Y,b)}+\xi(2Y,b)\overline{\xi^3(2Y,b)})i^{\T(2Yc+2bd)}\\
&=&2^{2m+2}\sum_{Y\in\Te}\sum_{b\in\Fq^*}(1+(-1)^{\tr(yb^{-1/3}+1)})(-1)^{\tr(ys+bd)}\\
&=&2^{2m+2}\sum_{b\in\Fq^*}(-1)^{\tr(bd+1)}\sum_{Y\in\Te}(-1)^{\tr(y(b^{-1/3}+s))}\\
&=& 2^{3m+2}(-1)^{\tr(s^{-3}d+1)}
\end{eqnarray*}
At last,
\begin{eqnarray*}
\mathbf{F}(c,d)_4&=&\xi^3(0,0)\overline{\xi(0,0)}+\xi(0,0)\overline{\xi^3(0,0)}
=2^{4m+1}.
\end{eqnarray*}
Adding $\mathbf{F}(c,d)_1,\mathbf{F}(c,d)_2,\mathbf{F}(c,d)_3$,
and $\mathbf{F}(c,d)_4$ up will complete the proof.

\end{proof}\qed

The following result can be proved similarly as above.
\begin{lemma}\label{lemma:N_8_3}
Let $c\in 2R$ and $d\in\Fq.$ We have
\begin{eqnarray*}
  \mathbf{F}(c,d)
              &=& \left\{
                \begin{array}{ll}
                  2^{3m+2}(3\cdot2^{m-1}-1), &\uif d=0;\\
                  2^{3m+2}(2^{m-1}-1),      &\uif d\neq 0.\\
                \end{array}
                \right.
\end{eqnarray*}
\end{lemma}

\bibliographystyle{plain}
\bibliography{Ref_DG}
\end{document}